\documentclass[a4paper,10pt,reqno]{amsart}        
\usepackage{verbatim}       
\usepackage{amssymb}       
\usepackage{mathrsfs}       
\usepackage{mathtools} 
\usepackage{enumerate}       
\usepackage[active]{srcltx}       
\usepackage{filemod}
\usepackage{filemod}

\usepackage{amsrefs}
\usepackage{thmtools}

\usepackage{todonotes}
\numberwithin{equation}{section}       
       
 \makeatletter      
  \@namedef{subjclassname@2020}{%     
  \textup{2020} Mathematics Subject Classification}    
   \makeatother    
       
\usepackage{t1enc}       
\usepackage[utf8x]{inputenc}

\newtheorem{theorem}{Theorem}[section]        
\newtheorem{lemma}[theorem]{Lemma}       
\newtheorem{obs}[theorem]{Observation}       
\newtheorem{observation}[theorem]{Observation}       
\newtheorem{problem}[theorem]{Problem}       
       
\newtheorem{prop}[theorem]{Proposition}       
    
\newtheorem{corollary}[theorem]{Corollary}       
\newtheorem{conj}[theorem]{Conjecture}       
       
\newtheorem*{theorem*}{Theorem}
\newtheorem*{corollary*}{Corollary}
\newtheorem*{example*}{Example}
     
\newcommand{\cf}{\operatorname{cf}}

\theoremstyle{definition}       
\newtheorem{definition}[theorem]{Definition}       
\newtheorem*{remark*}{Remark}
\newtheorem{remark}[theorem]{Remark} 

\newcommand{\mc}[1]{\mathcal{#1}}       
       
\newcommand{\mbb}[1]{\mathbb{#1}}       
\newcommand{\mb}[1]{\mathbf{#1}}       
       
\newcommand{\setm}{\setminus}       
\newcommand{\empt}{\emptyset}       
\newcommand{\subs}{\subset}       
\newcommand{\dom}{\operatorname{dom}}       
\newcommand{\ran}{\operatorname{ran}}       
\def\<{\left\langle}       
\def\>{\right\rangle}

\newcommand{\gamexyb}[3]{G_{#3}(#2,#1)}
\newcommand{\gamexb}[2]{G_{#2}(#1)}
\newcommand{\wingamexyb}[4]{#1\uparrow{G_{#4}(#3,#2)}}

\newcommand{\wingamexb}[3]{#1\uparrow{G_{#3}(#2)}}

\newcommand{\bocox}[3]{\operatorname{\mathbf B}_{#3}(#1,#2)}
\newcommand{\boco}[2]{\operatorname{\mathbf B}(#1,#2)}

\newcommand{\boou}[3]{\operatorname{B}_{#3}(#1,#2)}

\newcommand{\ps}{\operatorname{ps}}
\newcommand{\we}{\operatorname{w}}
\newcommand{{\llambda}}{\operatorname{sm}}
\newcommand{\sm}{\operatorname{sm}}
\newcommand{\psn}{\operatorname{\psi w_0}}
\newcommand{\pso}{\operatorname{\psi w_1}}
\newcommand{\pst}{\operatorname{\psi w_2}}
\newcommand{\psii}{\operatorname{\psi w_i}}

\newcommand{\oset}{\operatorname{o}}

\newcommand{\mbbI}{[0,1]}

\usepackage{xcolor}

\title[Gamification of the 
$T_0$-pseudoweight]{Gamification of the 
$T_0$-pseudoweight via cut-and-choose games on topological spaces}

\author[L. Chiozini]{Lucas Chiozini}
\address{Universit\`a degli Studi di Palermo}
\email{lucas.chiozini.souza@gmail.com}

\author[T. Csern\'ak]{Tam\'as Csern\'ak}
\address{University of Pannonia}
\email{tamas@csernak.hu}

\author[L. Soukup]{Lajos Soukup}

\address{Hun-Ren Alfréd Rényi Institute of Mathematics}
\email{lsoukup@gmail.com}



\begin{document}       

\date{\today}

\begin{abstract}
        We study transfinite cut-and-choose games on $T_0$
         spaces, introducing the {\em point-separating number} 
         $\ps(X)$ and the {\em set membership number} ${\sm}(X)$ as ordinal-valued 
         invariants measuring the minimal length of a game in which a Seeker can 
         determine a hidden point or subset. A central motivating question is which 
         countable ordinals can occur as the value of $\ps(X)$, in particular whether
          any countable ordinal  can arise. These invariants are inspired by Scott's 
          $T_0$-pseudoweight $\psn$. 
          Fundamental inequalities are established relating $\ps(X)$, ${\sm}(X)$, $\psn(X)$, 
          and $|X|$, including the sharp bounds $|X|\le 2^{\ps(X)}$ and 
          $\psn(X)\le 2^{<\ps(X)}$. The values of these invariants are computed for familiar spaces such as Cantor cubes, 
          powers of the Alexandroff double arrow space, and certain stationary subsets of cardinals. Finally, we further investigate their behavior under 
          topological sums and products, revealing the striking contrast between 
          $\ps$ and ${\sm}$. 

          For metric spaces, we determine that $\ps(X)=\log|X|$.
          However, we do not know such computation for ${\sm}(X)$;
          we can only assert that ${\sm}(X)$ may be arbitrarily large.
          
        Finally, we highlight another 
           open problem:  whether these games are always determined. 
\end{abstract}
\maketitle    

\section{Introduction}

We consider the following transfinite cut-and-choose game played on a $T_0$ topological space $X$ by two players, the \emph{Seeker} and the \emph{Hider}.  
At the beginning of the game, the Hider secretly selects a point $p \in X$.  
At the $\alpha$-th move, the Seeker chooses an open set $U_\alpha \subseteq X$, and the Hider responds by indicating whether  $p \in U_\alpha$.  
The game terminates once the Seeker can uniquely determine the point $p$.  

The \emph{point-separating number} of the space $X$ 
denoted by $\ps(X)$ is the least ordinal $\beta$ such that the Seeker has a strategy guaranteeing termination of the game in at most $\beta$ steps, regardless of the choice of $p$.  

It is clear that $\ps(\mathbb{R})=\ps(\mathbb{Q})=\omega$. More generally, $\ps(X)=\omega$ 
for every infinite metric space $X$ with $|X|\leq 2^{\omega}$.  
However, the value of $\ps(X)$ is not necessarily a cardinal: as we will see in Theorem~\ref{tm:double-arrow-power},  
\begin{displaymath}
\ps(\mathbb{A}^{n}) = \omega \oplus {n}
\end{displaymath}
for each   $1\le {n} \le {\omega}$,  
where $\mathbb{A}$ denotes the Alexandroff double arrow space, and $\oplus$ denotes ordinal addition.

Since we did not know of any space $X$ with $\omega \oplus \omega < \ps(X) < \omega_1$,  
we initiated a systematic investigation of this game and its variants,  
with the goal of understanding their behavior, identifying spaces for which the value of $\ps$ can be determined in advance,  
and computing $\ps$ for classical spaces such as the Suslin line, Cantor cubes, and ${\omega}^*$.

\subsection*{Historical background.}
In \cite{Scott86} Scott introduced the  notion of {\em $T_i$-pseudoweight $\psii(X)$}  of a 
$T_i$ topological space $X$  for $i\le2$ as follows:
\begin{displaymath}\label{df:psi_i}
    \psii(X)=\min\{\we(\rho):\text{$\rho\subset \tau_X$ is a $T_i$ topology on $X$}\}.
\end{displaymath}
For infinite $X$,  $\psii(X)$ is clearly the minimum of the cardinalities of $T_i$-separating families of open sets, 
and we have  
\begin{equation}\notag%\label{eq:1}
\text{$\psn(X)\le \pso(X)=\operatorname{\psi w}(X)\le \pst(X)\le \we(X)$,
}
\end{equation}
and, by \cite{Scott86}*{Fact 0.3 (c)},
\begin{equation}\notag%\label{eq:2}
|X|\le 2^{\psn(X)}.
\end{equation}

\subsection*{Formal definition of the games.}

Next, we introduce two transfinite cut-and-choose games that formalize and extend the Seeker–Hider game described above.  
The first of these is not new—it provides a precise 
game-theoretic formulation of the informal game introduced earlier—while both are naturally connected to the notion of $T_0$-pseudoweight.

If $X$ is a topological space, $U\subs X$ and  $i\in 2$, let  
\begin{displaymath}
\boou{U}{i}{X}=\left\{\begin{array}{ll}
{U}&\text{if $i=1$,}\\
{X\setm U}&\text{if $i=0$.}
\end{array}\right..
\end{displaymath}
If  $\mathbf U\in {}^{\beta}\mc P(X)$ and $\mathbf i\in {}^{\beta}2$  are functions for some ordinal ${\beta}$, let
\begin{displaymath}
    \bocox{\mathbf U}{\mathbf i}{X}=\bigcap_{{\alpha}<{\beta}}
    \boou{\mathbf U({\alpha})}{\mathbf i({\alpha})}{X}.
\end{displaymath}
We write $\boco{\mathbf U}{\mathbf i}$ for $\bocox{\mathbf U}{\mathbf i}{X}$ if $X$ is clear from the context. 

\begin{definition}\label{df:game_YX}
    Assume that $X$ is a topological space, $Y\subs X$, and ${\beta}$ is an ordinal.
    We define the {\em set membership} game  $\gamexyb{X}{Y}{{\beta}}$ and  
    the {\em point separating} game $\gamexb X{\beta}$ as follows. 

    The gameplay is the same for both games; the difference lies in how the winner  
    is determined.
    There are two players:  the {\em Seeker} and the {\em Hider}. 
    The gameplay  consists of ${\beta}$-many moves.

In the $\alpha$-th move of the game,  
\begin{enumerate}[(i)]
    \item the Seeker chooses an open set $\mathbf{U}(\alpha) \in \tau_X$, and  
    \item the Hider then   declares  
\begin{displaymath}
    \mathbf{i}(\alpha) =
    \begin{cases}
        1 & \text{if the Hider chooses } \mathbf{V}(\alpha) = \mathbf{U}(\alpha), \\
        0 & \text{if the Hider chooses } \mathbf{V}(\alpha) = X \setminus \mathbf{U}(\alpha).
    \end{cases}
\end{displaymath}    
\end{enumerate}
The \emph{state} of the game after $\alpha^*$ moves is the set  
\begin{displaymath}
\mathbf{W}(\alpha^*) = \bigcap_{\alpha < \alpha^*} \mathbf{V}(\alpha) 
= \bocox{\mathbf{U}\restriction \alpha^*}{\mathbf{i}\restriction \alpha^*}{X}.
\end{displaymath}
The \emph{outcome} of the game is defined by  
\begin{displaymath}
\mathbf{W} = \mathbf{W}(\beta).
\end{displaymath}
In the game $\gamexyb{X}{Y}{{\beta}}$, the Seeker wins 
iff 
\begin{displaymath}
\text{$\mb W\subs Y$ or $\mb W\cap Y=\empt.$}
\end{displaymath}    
In the game $\gamexb{X}{{\beta}}$, the Seeker wins iff    
\begin{displaymath}
\text{$|\mb W|\le 1$.}
\end{displaymath}

We write $\wingamexyb{S}{X}{Y}{{\beta}}$ \ (respectively, $\wingamexyb{H}{X}{Y}{{\beta}}$)  iff 
the Seeker (respectively, the Hider) has a winning strategy in the game 
$\gamexyb{X}Y{{\beta}}$.

Likewise, 
 $\wingamexb{S}{X}{{\beta}}$ \ (respectively, $\wingamexb{H}{X}{{\beta}}$)  means that  
the Seeker (respectively, the Hider) has a winning strategy in the game $\gamexb{X}{{\beta}}$.

Define the {\em set membership number} ${\sm}(Y,X)$ of a subspace $Y$ in a $T_0$ space  $X$ and the 
{\em set membership number} ${\sm}(X)$ of a  $T_0$ space  $X$ 
as follows: 
\begin{align*}
{\sm}(Y,X)=&\min\{{\beta}:\wingamexyb{S}{X}{Y}{{\beta}}\},\\
{\sm}(X)=&\sup\{{\sm}(Y,X):Y\subs X\}.
\end{align*}
Define  the {\em point separation number} $\ps(X)$ of a $T_0$ space $X$ as follows:  
\begin{displaymath}
    \ps(X)=\min\{{\beta}:\wingamexb{S}{X}{{\beta}}\}.
    \end{displaymath}
\end{definition}
It is straightforward to verify that this definition of $\ps(X)$ coincides with the one given at the beginning of the introduction.

Clearly, ${\sm}(Y,X)$, ${\sm}(X)$ and $\ps(X)$  are defined provided 
$X$ is $T_0$, and 
\begin{equation}\notag%\label{eq:}
\text{${\sm}(Y,X)\le {\sm}(X)\le ps(X)\le \psn(X)%\le  \min(|X|,{\rm w}(X))
.$}
\end{equation}

\subsection*{Basic inequalities.}
In Section \ref{sc:ineq} we  investigate the relationships between 
$\ps(X)$, ${\sm}(X)$, $\psn(X)$ and $|X|$. 
The simplest question is whether the value of $\ps(X)$ bounds the value of $\psn(X)$, 
and whether the value of ${\sm}(X)$ bounds $\ps(X)$.

By \cite{Scott86}*{Fact 0.3 (c)}, we have 
$
|X|\le 2^{\psn(X)}.$   
Since $\ps(X)\le \psn(X)$, in Theorem \ref{tm:ps2psn}(a) we obtain 
the following stronger result:

\begin{theorem*}
    If $X$ is a $T_0$-space, then $|X|\le 2^{\ps(X)}$.
\end{theorem*}

It is important to emphasize that, unlike $\psii(X)$,   the values  $\ps(X)$ 
and ${\sm}(X)$  
are not necessarily cardinals. Hence, to formulate the next result, we need 
to introduce the following notation: if ${\beta}$ is an ordinal, let 
$$2^{<{\beta}}
=\sup \{|{}^{{\alpha}}2|:{\alpha}<{\beta}\}.$$
So $2^{<{\omega}}={\omega}$, but $2^{<{\omega}\oplus 1}=2^{\omega}$.
%where $\oplus$ denotes the ordinal addition. 
In Theorem \ref{tm:ps2psn}(b) we obtain a  bound for $\psn(X)$. 

    \begin{theorem*}
        If $X$ is a $T_0$-space, then $\psn(X) \leq 2^{<\ps(X)}$. 
        \end{theorem*}

        By Theorem      \ref{tm:ps=kappa} the 
 space $X=\<D(2)^{\kappa}, {\tau}_{<{\kappa}}\>$ serves in the following example,  where ${\tau}_{<{\kappa}}$ denotes the topology generated by the 
 $G_{<{\kappa}}$-sets.  It shows that  
the inequality $\psn(X) \leq 2^{<\ps(X)}$ is sharp: 
\begin{example*}
For each regular cardinal ${\kappa}$ there is a 0-dimensional $T_2$ 
space $X$ with $\ps(X)={\kappa}$ and $\psn(X)=2^{<{\kappa}}$.
\end{example*}

    Under GCH, $2^{<{\kappa}}={\kappa}$ for each infinite cardinal ${\kappa}$,
    so in this case    
    $\ps(X) = {\kappa}$ implies $ \psn(X)=\ps(X)$.   
    However, in Theorem \ref{tm:double-arrow-power} we construct  ZFC examples when  $\ps(X)$ is an ordinal, 
    and it is strictly smaller than 
    $\psn(X)$:
    
    \begin{example*}
    $\ps(\mathbb A) =  {\sm}(\mbb A)= \omega \oplus 1 < 2^{\omega} = \psn(\mathbb A)$, 
    where $\mathbb A$ denotes the Alexandroff double arrow space. 
    \end{example*}

  In Section \ref{sc:ineq}, we establish  a connection between the values of $\ps(X)$ and ${\sm}(X)$.
To state this result, we introduce the  following terminology. 

Let $\alpha$ be an ordinal and $\kappa$ a cardinal. 
We say that $\alpha$ is {\em $\kappa$-decomposable} if $\alpha$ can be partitioned into $\kappa$ subsets, each of which has order type $\alpha$.
Observe that $\alpha$ is  $\kappa$-decomposable if and only if
${\beta}\oplus {\kappa}\le {\alpha}$ for each ${\beta}<{\alpha}$
if and only if ${\alpha}={\kappa}\otimes {\gamma}$ for some ordinal ${\gamma}$.

In Theorem \ref{tm:pssmlog} we prove the following.

  \begin{theorem*}
     $\ps(X)=\max({\sm}(X),\log(|X|))$ provided $\sm(X)< \log(|X|)$ or   
    ${\sm}(X)$ is  $\log(|X|)$-decomposable.
    \end{theorem*}

    In particular, if $|X|\le 2^{\omega}$, then  $\ps(X)=\max({\sm}(X),\log(|X|))$
    provided ${\sm}(X)$ is a limit ordinal.
    Without any extra assumption the equality of the previous 
    theorem does not hold: 
 by  Theorem \ref{tm:double-arrow-power},  
 \begin{displaymath}
 \text{
 $\ps(\mbb A^2)={\omega}\oplus 2$ and 
 $\max({\sm}(\mbb A^2),\log(|
 \mbb A^2|))={\omega}\oplus 1$.
 }
\end{displaymath}
    The set membership number of a discrete space is $1$,  
    and thus the value of ${\sm}(X)$ places no restriction on $|X|$.  
    Moreover, ${\sm}(X)$ alone does not bound the size of $X$ even for crowded (i.e. dense-in-itself) spaces,  
    nor does it do so when $\Delta(X)$ is large.  
    Indeed, by Corollary \ref{cr:subm} we have the following examples.

\begin{example*}  
For each infinite cardinal $\kappa$, there exists a $0$-dimensional $T_2$ space $X$ such that  
\begin{displaymath}
|X| = \Delta(X) = \kappa, \quad\ps(X)=\log({\kappa}) \quad \text{and} \quad {\sm}(X) = 2.
\end{displaymath}  
\end{example*}

However, if $\oset(X)<2^{|X|}$, then the following statement 
(see Theorem \ref{tm:Xosm}) provides a lower bound for 
$\sm(X)$.
\begin{theorem*}
If $X$ is a $T_0$ space, then 
\begin{displaymath}
\oset(X)^{2^{\sm(X)}}\ge 2^{|X|}.
\end{displaymath}
\end{theorem*}

In particular, under GCH, we obtain the following (see 
Corollary \ref{cr:smlarge}):
\begin{corollary*}[GCH]
    If $X$ is an infinite  $T_0$ space with $|X|=2^{\we(X)}$,
    then $\sm(X)=\we(X)$. 
\end{corollary*}

\subsection*{On the range of $\ps$ and ${\sm}$.}

The starting point of our investigation was the
 following question: determine the 
class of the possible values of $\ps$ and ${\sm}$.
In particular, we are interested in identifying those countable ordinals that can occur 
as the value of $\ps$.  
At present, we have only partial results.  
In particular,  
Theorem \ref{cor:addition} yields the following: 

\begin{theorem*}
For each ordinal $ {\alpha}\le{\omega}^2$, there exists a
0-dimensional $T_2$ space $X$ with $\ps(X)={\alpha}$.   
\end{theorem*}

\begin{restatable}{problem}{alphaleomegaone}\label{pr:alphaleomega1}
    Does there exist a topological space $X$ with ${\omega}^2<\ps(X)<{\omega}_1$? 
\end{restatable}

\subsection*{$\ps$ and ${\sm}$ for classical spaces}
In Section \ref{sc:familiar} we compute $\ps$ and ${\sm}$ for some 
well-studied spaces. 
In Theorem \ref{tm:ps=kappa} we compute $\ps$ and ${\sm}$ for the Cantor cubes. 
\begin{theorem*}
$\ps(D(2)^{\kappa})=
{\sm}(D(2)^{\kappa})={\kappa}$ 
for each infinite cardinal ${\kappa}$.
\end{theorem*}

In Theorem \ref{tm:double-arrow-power} we compute the $\ps$ and $\sm$
for powers of the Alexandroff double arrow space $\mbb A$.
In particular, we obtain the following results:
\begin{theorem*}
$\ps(\mbb A^n)={\omega}\oplus n$
and ${\sm}(\mbb A^n)={\omega}\oplus 1$ for $n<{\omega}$,
and $\ps(\mbb A^{\mu})={\sm}(\mbb A^{\mu})={\omega}\oplus {\mu}$
for any infinite cardinal ${\mu}$.
\end{theorem*}

In Theorem  \ref{pr:stat} we proved that ${\sm}(S)={\kappa}$
whenever ${\kappa}$ is a regular cardinal and $S$ 
is a stationary subset of ${\kappa}$. However,   so far
we have not been able to compute  $\ps$ and ${\sm}$ for the Suslin lines or for ${\omega}^*$.

\smallskip

\subsection*{Topological sums and products.}

 In Section \ref{sc:sum-prod}
 we investigate the behavior of cardinal invariants $\ps$
 and ${\sm}$ under topological sums and products. 

 The behavior of $\ps$ under topological sums is particularly well-behaved. 
In Theorem \ref{tm:pssumgen}
 we show that for each infinite cardinal ${\kappa}$,  
     \begin{displaymath}
     \ps\big(\sum \{X_{\sigma}:{\alpha}<{\kappa}\}\big)=
     \log({\kappa})\cdot \sup\{\ps(X_{\alpha}:{\alpha}<{\kappa})\}.
     \end{displaymath}

The behavior of ${\sm}$ on sum is completely different as
shown by the next two theorems. 
In Theorem \ref{tm:sumsm} we show
    \begin{displaymath}
    {\sm}(\sum_{{\alpha}<{\kappa}}X_{\alpha})\le \sup_{{\alpha}<{\kappa}}|X_{\alpha}|
    \end{displaymath}
    provided that every $X_{\alpha}$ is $T_1$. 
In Theorem \ref{tm:sm-limit}
we show 
        $${\sm}(\sum_{{\alpha}\in I}X_{\alpha})\le 
        (\sup_{{\alpha}\in I}{\sm}(X_{\alpha}))\oplus1.$$
      
 These two results raise the following problem:      
      \begin{restatable}{problem}{exactsum}\label{cr:product2}
         Is it true that ${\sm}(\sum_{{\alpha}<{\lambda}}X_{\alpha})= \sup\{{\sm}(X_{\alpha}):{\alpha}<{\lambda}\}$?
         \end{restatable}

         Next, we turn our attention to products of spaces.

         In Theorem \ref{pro:ps-prod2}      
  we show that if $\cf(\sup\{\ps(X_{\alpha}:{\alpha}<{\lambda})\})\ge {\lambda}$, then 
$$\ps(\prod_{{\alpha}<{\lambda}}X_{\alpha})=\sup\{\ps(X_{\alpha}:{\alpha}<{\lambda})\}).$$

In contrast, the behavior of ${\sm}$ under products is very different.

\begin{example*}
For each cardinal ${\kappa}$ there is a  space $X$ such that 
${\sm}(X)=2$ and ${\sm}(X\times X)\ge {\omega}$. 
\end{example*}

\begin{proof}
Indeed, by \cite{JuSoSz06}*{Theorem 4.1}, there is a  submaximal space  $X$ with $|X|=\Delta(X)={\kappa}$.
Then ${\sm}(X)=2$, $X\times X$ is resolvable, so ${\sm}(X\times X)\ge {\omega}$
by Theorem \ref{tm:resolvable}. 
\end{proof}

\begin{restatable}{problem}{productsm} Is it true that 
    for each cardinal ${\kappa}$ there is a  space $X$ such that 
    ${\sm}(X)=2$ and ${\sm}(X\times X)\ge {\kappa}$.
\end{restatable}

\subsection*{Metric spaces.}
      
         By \cite{Scott86}*{Theorem 1.7}, we have 
         $\psn( X)  = \log |X |$ for every metric space $X$. 
In Section \ref{sc:developable} we prove  Theorem \ref{tm:quasi}, which  immediately implies the following result. 
\begin{theorem*}
    For each metric space $X$, $\ps(X)=\log(|X|)$.
\end{theorem*}

We do not have a similar closed formula to compute ${\sm}(X)$ for a metric space $X$,
but in Theorem \ref{tm:metric-sm}
we show that for each cardinal ${\kappa}$ there is a metric space with ${\sm}(X)>{\kappa}.$

\section{Inequalities}\label{sc:ineq}

\begin{theorem}\label{tm:ps2psn}   
    Assume that $X$ is a $T_0$ space.
    
    \noindent {(a) $|X|\le 2^{\ps(X)}$.}
    
     \noindent    (b) ${\psn}(X)\le 2^{<\ps(X)}$.
    In particular, if
    $\ps(X)={\omega}$, then ${\psn}(X)={\omega}$ as well.
    \end{theorem}
    
    \begin{proof}
    Let ${\sigma}$ be a winning strategy of Seeker in the game $\gamexb{X}{{\ps(X)}}$.

    To prove  (a), fix $x\in X$ and consider the following play of the game  $\gamexb{X}{{\ps(X)}}$. The Seeker follows the strategy ${\sigma}$, while  
    the Hider respond so that $\mb i({\alpha})=1$ iff $x\in \mb U({\alpha})$. 
    Consequently,  $x\in \bocox{\mathbf U}{\mathbf i}{X}$.  
 Since the Seeker wins,   we have $$\bocox{\mathbf U}{\mathbf i}{X}=\{x\}.$$
    Thus,  the sequence $\mathbf i$ uniquely  determines $x$. 
    As there are exactly $2^{\ps(X)}$ many such  
    sequences $\mathbf i$, the  proof of (a) is complete.

    To show (b)
    we claim that 
    \begin{displaymath}\tag{$\star$}
    \forall \{a,b\}\in {[X]}^{2}\ \exists U\in \ran({\sigma})\text{ such that } |U\cap \{a,b\}|=1.
    \end{displaymath}
    Assume, on the contrary, that  this is not true for some $\{a,b\}\in {[X]}^{2}$.
    Then let Hider play the following strategy:
    \begin{displaymath}
    \mathbf i({\alpha})=1 \text{ iff } \{a,b\}\subs \mathbf U({\alpha}).
    \end{displaymath}  
    So $\mathbf i({\alpha})=0$ iff $\{a,b\}\subs X\setm \mathbf U({\alpha})$.
    Hence, $\{a,b\}\subs \boco{\mathbf U}{\mathbf i} $ and so 
    Hider wins. So ${\sigma}$ is not a winning strategy for the Seeker.  Contradiction. So $(\star)$ holds.

    Since $|\ran({\sigma})|\le 2^{<\ps(X)}$, the proof of (b) is complete. 
\end{proof}

    \begin{theorem}\label{tm:pssmlog}
        Let $X$ be a $T_0$-space and let ${\delta}=\max({\sm}(X),\log(|X|))$.
        If ${\delta}$ is  $\log(|X|)$-decomposable,  then 
        $\ps(X)={\delta}$.
        \end{theorem}

\begin{proof}
 By Theorem   \ref{tm:ps2psn}(a), we have $\ps(X)\ge\max({\sm}(X),\log(|X|))={\delta}$. 

To show that $\ps(X)\le {\delta}$, 
choose
a family  $\{A_{\alpha}:{\alpha}<\log(|X|)\}\subs \mc P(X)$ such that 
for every $\{p,q\}\in {[X]}^{2}$, there exists ${\alpha}$
with $|\{p,q\}\cap A_{\alpha}|=1$.

Since ${\delta}$ is  $\log(|X|)$-decomposable, we can 
partition ${\delta}$ into $\log(|X|)$ many pieces, each of order type ${\delta}$, 
say $\{I_\alpha : \alpha < \log(|X|)\}$.  
For each $\zeta\in\delta$ let $\alpha_{\zeta}<\kappa$ be the unique index such that 
$\zeta\in I_{\alpha_{\zeta}}$.

    Informally, the Seeker’s strategy is to play $\kappa$ parallel games, using the coordinates in $I_{\alpha}$
    to determine whether the point selected by the Hider is in 
    $A_{\alpha}$ or not.
    
    Formally, for each $\alpha<\kappa$ let $\sigma_{\alpha}$ be a winning strategy for 
    the Seeker in the game $\gamexyb{X}{A_{\alpha}}{\delta}$.

    Define the strategy $\sigma$ in $\gamexb{X}{{\delta}}$ for the Seeker as follows.

    Suppose the game has proceeded up to stage $\zeta$, so that 
    $\mathbf{U}\restriction \zeta$ and $\mathbf{i}\restriction \zeta$ are known.  
    Set
    \begin{displaymath}
    \mathbf{U}(\zeta) =
    \sigma_{\alpha_{\zeta}}\bigl(\mathbf{i}\restriction (I_{\alpha_{\zeta}} \cap \zeta)\bigr).
    \end{displaymath}

    After ${\delta}$ steps, denote $\mb W$ the outcome of the game.

Now, the pair $\<\mb U\restriction I_{\alpha}, \mb i\restriction I_{\alpha}\>$ codes a play
of the game $\gamexyb{X}{A_{\alpha}}{{\delta}}$ in which  the Seeker follows the strategy 
${\sigma}_{\alpha}$. 
Hence, the outcome of that game,  $\mb  W^{\alpha}=
    \bocox{\mb U\restriction I_{\alpha}}{\mb i\restriction I_{\alpha}}{X}$
is either a subset of $A_{\alpha}$ or disjoint to $A_{\alpha}$.
Consequently, $\mb W\subs \mb W^{\alpha}$ is also either a subset of $A_{\alpha}$ or disjoint to $A_{\alpha}$.

Since $\{A_{\alpha}:{\alpha}<\ps(|X|)\}$ is a point-separating family, it follows
that $|\mb W|\le 1$.  Therefore, using the strategy ${\sigma}$, the Seeker wins the game $\gamexb{X}{{\delta}}$. 
\end{proof}

\begin{corollary}\label{cr:smpslog}
    Let $X$ be a $T_0$-space. If ${\sm}(X)$ is a cardinal, then  
    $\ps(X)=\max({\sm}(X),\log(|X|))$.

\end{corollary}

If $\oset(X)<2^{|X|}$, then the following theorem gives a lower bound for 
$\sm(X)$.

\begin{theorem}\label{tm:Xosm}
If $X$ is a $T_0$ space, then 
\begin{displaymath}
\oset(X)^{2^{\sm(X)}}\ge 2^{|X|}.
\end{displaymath}
\end{theorem}

\begin{proof}
    By definition of $\sm(X)$,    
 for each  subset $Y\subs X$,  
the Seeker has a winning strategy ${\sigma}_Y$ in the game 
$\gamexyb{X}{Y}{\sm(X)}$.

Each strategy  ${\sigma}_Y$ is a function mapping the tree  
$\bigcup\{{^{\alpha}}2:{\alpha}<\sm(X)\}$
 into the collection of open subsets of $X$.
%Clearly, $|{\tau}_X|\le 2^{\we(X)}=2^{\kappa}$.  
Hence, 
\begin{equation}\label{eq:numsigma}
|\{{\sigma}_Y:Y\subs X\}|\le  \oset(X)^{2^{<\sm(X)}}\le \oset(X)^{2^{\sm(X)}}.
\end{equation}

For each $Y\subs X$, define 
a function ${\varphi}_Y:{}^{\sm(X)}2\to 2$ by
\begin{displaymath}
\text{${\varphi}_Y(f)=1$ iff $A_{{\sigma}_Y,f}\subs Y$},
\end{displaymath} 
where $A_{{\sigma}_Y,f}$  is the outcome of the game $\gamexyb{X}{Y}{\sm(X)}$  in which the
Seeker follows the strategy  ${\sigma}_Y$ and Hider plays $f({\zeta})$ at the ${\zeta}$th stage of the game.
Since ${\sigma}_Y$ is a winning strategy in $\gamexyb{X}{Y}{\sm(X)}$, we 
also have  
\begin{displaymath}
    \text{${\varphi}_Y(f)=0$ iff $A_{{\sigma}_Y,f}\cap Y=\empt$.}
\end{displaymath}
Hence, 
\begin{equation}\label{eq:decode}\tag{$\dag$}
Y=\bigcup\{A_{{\sigma}_Y,f}:f\in {}^{\sm(X)}2\land  {\varphi}_Y(f)=1\}.
\end{equation}
Thus,  
\begin{equation}\label{eq:oneone}
\text{the mapping $Y\mapsto \<{\sigma}_Y,\varphi_Y\>$ is injective.}
\end{equation}
The number of possible ${\varphi}_Y$ is at most
$2^{2^{\sm(X)}}$, since $\varphi_Y$  is a function with 
 domain  ${}^{\sm(X)}2$. Therefore, we  have
\begin{equation}\label{eq:phinum}
|\{\varphi_Y:Y\subs X\}|\le 2^{2^{\sm(X)}}.
\end{equation}
Combining  \eqref{eq:numsigma}, \eqref{eq:phinum} and \eqref{eq:oneone}, we obtain 
\begin{equation}\label{eq:final}\notag
2^{|X|}=|\mc P(X)|\le |\{\<{\sigma}_Y,\varphi_Y\>:Y\subs X\}|\le
\oset(X)^{2^{\sm(X)}}\cdot
2^{2^{\sm(X)}}=\oset(X)^{2^{\sm(X)}},
\end{equation}
which completes the proof.
\end{proof}

\begin{corollary}\label{tm:smlarge}
Assume that ${\kappa}$ is an infinite cardinal such that $2^{\lambda}\le {\kappa}$ for each ${\lambda}<{\kappa}$. 
If $X$ is a $T_0$ space with       $|X|=2^{{\kappa}}$ and   $\we(X)={\kappa}$, then $\sm(X)={\kappa}$ as well.
\end{corollary}

\begin{proof}
    Since $\sm(X)\le \we(X)$, we have $\sm(X)\le {\kappa}$.
    
    To show the reverse inequality, first observe that $\we(X)={\kappa}$
implies $\oset(X)\le 2^{\kappa}$. 

Hence, using Theorem \ref{tm:Xosm}
we have 
\begin{displaymath}
2^{{\kappa}\cdot 2^{\sm(X)}}\ge 2^{2^{\kappa}}.
\end{displaymath} 
If $\sm(X)<{\kappa}$, then $2^{\sm(X)}\le {\kappa}$,
so the LHS  of the previous inequality is $2^{\kappa}$ which is smaller than 
$2^{2^{\kappa}}$. Hence, $\sm(X)\ge{\kappa}$, which completes the proof of the Corollary.
\end{proof}

\begin{corollary}[GCH]\label{cr:smlarge}
If $X$ is an infinite  $T_0$-space with $|X|=2^{\we(X)}$,
then $\sm(X)=\we(X)$. 
\end{corollary}

\begin{proof}
If GCH holds, then the previous theorem can be applied for arbitrary 
${\kappa}=\sm(X)$.
\end{proof}

\section{Games on topological sums and products}\label{sc:sum-prod}

\subsection*{Topological sum}

In \cite{Scott86}*{Theorem 1.1} Scott proved the following:\\
 {\em If    $\psn(X_{\alpha}) \le {\kappa} $ for
 ${\alpha} < 2^{\kappa}$, then  $\psn(\sum_{{\alpha}<2^{\kappa}}X_{\alpha})\le{\kappa}$. 
 } The next theorem is a similar  statement
 concerning $\ps$.

\begin{theorem}\label{tm:pssumgen}
Let $\<X_{\alpha}:{\alpha}<{\kappa}\>$ be $T_0$ spaces.
If $\max(\log(\kappa),\sup\{\ps(X_{\alpha}) : \alpha<\kappa\})$ is infinite,
then 
\begin{displaymath}
    \ps\!\Big(\sum\nolimits_{\alpha<\kappa} X_{\alpha}\Big) =
    \max(\log(\kappa),\sup\{\ps(X_{\alpha}) : \alpha<\kappa\}).
\end{displaymath}
\end{theorem}
    
    \begin{proof} By Theorem \ref{tm:ps2psn}(a), the inequality ``$\ge$'' follows immediately.

        To show the reverse inequality, 
    let %$X=\sum_{\alpha<\kappa} X_{\alpha}$ and write
    \begin{displaymath}
        X=\sum_{\alpha<\kappa} X_{\alpha}\text{ and }        {\delta}=\max(\log(\kappa),\sup\{\ps(X_{\alpha}) : \alpha<\kappa\}).
\end{displaymath}
Since ${{\delta}}$ is infinite, it has a partition ${{\delta}}=E\cup O$
such that $\operatorname{tp}(O)=\log(\kappa)$
and $\operatorname{tp}(E)=\sup\{\ps(X_{\alpha}) : \alpha<\kappa\}$.
    
    Fix a family $\mathcal A=\{A_{\zeta} : \zeta\in \log({\kappa})\}\subs \mc P({\kappa})$ such that 
    for each $\{\alpha,\beta\}\in[\kappa]^2$ there exists $\zeta<{\delta}$
    with $\alpha\in A_{\zeta}$ and $\beta\notin A_{\zeta}$.
    
    For each $\alpha<\kappa$, let $\sigma_{\alpha}$ be a winning strategy of 
    the Seeker in the game $\gamexb{X_{\alpha}}{{\operatorname{tp}(E)}}$.

    Informally, the Seeker play the following   
    strategy   in the game $\gamexb{X}{{\delta}}$:
    on  coordinates in $O$, the Seeker plays to determine the unique $\alpha$ such that 
    $X_{\alpha}$ contains the Hider’s chosen point.  
    On coordinates in $E$, the Seeker plays a parallel copy of the winning  strategy $\sigma_{\alpha}$ 
    inside each $X_{\alpha}$.

    \smallskip
    The  formal definition of the Seeker's strategy ${\sigma}$ is as follows.  
    If $\xi\in O$, then the Seeker’s move is
    independent of the Hider’s previous moves:
    \begin{displaymath}
    \sigma(\mathbf i\restriction \xi)=\bigcup\{X_{\alpha} : 
    \alpha\in A_{\operatorname{tp}(O\cap {\xi})}\}.
    \end{displaymath}
    If $\xi\in E$, then 
    \begin{displaymath}
    \sigma(\mb i\restriction \xi)=\bigcup_{\alpha<\kappa}\sigma_{\alpha}(\mathbf i\restriction E).
    \end{displaymath}
    
    \medskip
    \noindent
    We now verify that ${\sigma}$ is really a winning strategy for the Seeker in the game $\gamexb{X}{{\delta}}$.

The family $\mathcal{A}$ separates the points of $\kappa$, so the Seeker’s moves on coordinates in $O$ guarantee that there exists a unique $\alpha < \kappa$ such that
\begin{displaymath}
\bocox{\mathbf{U} \restriction O}{\mathbf{i} \restriction O}{X} \subseteq X_{\alpha}.
\end{displaymath}
For each $\xi \in E$, we have
\begin{displaymath}
\sigma(\mathbf{i} \restriction \xi) \cap X_{\alpha} = \sigma_{\alpha}(\mathbf{i} \restriction E),
\end{displaymath}
so the trace of the game on $X_{\alpha}$ using coordinates in $E$ is exactly a game in $\gamexb{X_{\alpha}}{\delta}$ where the Seeker follows the winning strategy $\sigma_{\alpha}$.  
Hence,
\begin{displaymath}
\bigl| \bocox{\mathbf{U} \restriction E}{\mathbf{i} \restriction E}{X} \cap X_{\alpha} \bigr| \le 1.
\end{displaymath}

It follows that 
\begin{displaymath}
|\bocox{\mathbf{U}}{\mathbf{i}}{X}| \le 1,
\end{displaymath}
so the Seeker wins the game $\gamexb{X}{\delta}$, as required.
    \end{proof}

\begin{restatable}{theorem}{tmsumsm}\label{tm:sumsm}
    \begin{displaymath}
    {\sm}(\sum_{{\alpha}<{\kappa}}X_{\alpha})\le \sup_{{\alpha}<{\kappa}}|X_{\alpha}| 
    \end{displaymath}
    provided that every $X_{\alpha}$ is $T_1$. 
    \end{restatable}

\begin{proof}
    Write  ${\lambda}=\sup_{{\alpha}<{\kappa}}|X_{\alpha}|$ and 
    $X=\sum_{{\alpha}<{\kappa}}X_{\alpha}.$ 
    We may assume that the underlying set of 
    $X_{\alpha}$ is contained in $\{{\alpha}\}\times {\lambda}$, and hence   
    $X\subs {\kappa}\times {\lambda}$. 

Fix $Y\subs X$.    Let the Seeker play the following strategy in  the game $\gamexyb{X}{Y}{{\lambda}}$.
    \begin{enumerate}[(1)]
    \item     If $\mb i({\zeta})=1$ for each ${\zeta}<{\eta}$, 
    then let $\mb U({\eta})=X\setm ({\kappa}\times \{{\eta}\})$.
    Since every $X_{\alpha}$ is  $T_1$, $\mb U({\eta})\in {\tau}_X$.
\item If $\mb i({\zeta})=0$ then let 
\begin{displaymath}
\mb U({\zeta}\oplus 1)=\bigcup\{X_{\alpha}:{\alpha}<{\kappa} \land \<{\alpha},{\zeta}\>\in Y\}.
\end{displaymath}
    \end{enumerate}
We claim that it is a winning strategy for the Seeker. 

Indeed, if $\mb i({\eta})=1$ for each ${\eta}<{\lambda}$,
then 
\begin{displaymath}
\bocox{\mb U}{\mb i}{X}=\bigcap_{{\eta}<{\lambda}}(X\setm ({\kappa}\times \{{\eta}\}))=\empt,
\end{displaymath}
so the Seeker wins. 

Assume know that
\begin{displaymath}
{{\zeta}}=\min\{{{\zeta}}':\mb i({{\zeta}}')=0\}
\end{displaymath} 
is defined. 
Then  $\mb U({\zeta})=X\setm ({\kappa}\times \{{\zeta}\})$ and 
\begin{displaymath}
    \bocox{\mb U\restriction {\zeta}\oplus 1}{\mb i\restriction {\zeta}\oplus 1}{X}
    \subset\boou{\mb U({\zeta})}{0}{X}=
 X\setm \mb U({\zeta})=X\cap ({\kappa}\times \{{\zeta}\}).
    \end{displaymath}
If $\mb i({\zeta}\oplus 1)=1$, then 
\begin{displaymath}
    \bocox{\mb U}{\mb i}{X}\subs (X\cap ({\kappa}\times \{{\zeta}\}))\cap\bigcup\{X_{\alpha}:{\alpha}<{\kappa} \land \<{\alpha},{\zeta}\>\in Y\}\subs Y .
    \end{displaymath}
    If $\mb i({\zeta}\oplus 1)=0$, then 
    \begin{displaymath}
        \bocox{\mb U}{\mb i}{X}\subs (X\cap ({\kappa}\times \{{\zeta}\}))\cap\bigcup\{X_{\alpha}:{\alpha}<{\kappa} \land \<{\alpha},{\zeta}\>\notin Y\}\subs X\setm Y .
        \end{displaymath}
So in both cases, the Seeker wins in  the game $\gamexyb{X}{Y}{{\lambda}}$.        
\end{proof}

    \begin{corollary}\label{tm:lambda-ps}
        Assume that ${\kappa}$ is an infinite cardinal.
        Then ${\sm}(D({\kappa})\times \mbb Q)={\omega}$ and 
        ${\sm}({\alpha}(D({\kappa})\times [0,1]))\le 2^{\omega}$.
        \end{corollary}

        \medskip

\begin{restatable}{theorem}{tmsmlimit}\label{tm:sm-limit}
    ${\sm}(\sum_{{\alpha}\in I}X_{\alpha})\le 
    \sup_{{\alpha}\in I}{\sm}(X_{\alpha})\oplus1$.
  \end{restatable}

\begin{proof} Write $X=\sum_{{\alpha}\in I}X_{\alpha}$.
Let $Y\subs X$. 
Write ${\delta}=\sup_{{\alpha}\in I}{\sm}(X_{\alpha})$.
Fix
 a winning strategy    ${\sigma}_{\alpha}$  of the Seeker in the game 
$\gamexyb{X_{\alpha}}{X_{\alpha}\cap Y}{{\delta}}$ for ${\alpha}\in I$.
Define the strategy ${\sigma}$ for the Seeker in 
$\gamexyb{X}{X\cap Y}{{\delta}\oplus 1}$ as follows. For ${\xi}<{\delta}$,
\begin{displaymath}
{\sigma}(\mathbf i\restriction {\xi})=\bigcup_{{\alpha}\in I}{\sigma}_{\alpha}(\mathbf i\restriction {\xi}).
\end{displaymath}
Then, for each ${\alpha}\in I$,
\begin{displaymath}
    \bocox{\mathbf U\restriction {\delta}}{\mathbf i\restriction {\delta}}X\cap X_{\alpha}\subs Y \text{ or }
    \bocox{\mathbf U\restriction {\delta}}{\mathbf i\restriction {\delta}}X\cap X_{\alpha}\cap  Y =\empt.
\end{displaymath}
In the ${\delta}$th step the Seeker asks:
 $$U({\delta})=\bigcup\{X_{\alpha}:\bocox{\mathbf U\restriction {\delta}}{\mathbf i\restriction {\delta}}X\cap X_{\alpha}\subs Y\}.$$
 If $\mathbf B({\delta})=1$, then $\bocox{\mathbf U}{\mathbf i}{X}\subs Y$.
 If $\mathbf B({\delta})=0$, then $\bocox{\mathbf U}{\mathbf i}X\cap Y=\empt$.
    \end{proof}

\subsection*{Topological Product}

In \cite{Scott86}*{Theorem 1.0}, Scott proved the following:  
\emph{If $\psn(X_{\alpha}) \le \kappa$ for
${\alpha} < \kappa$, then 
$\psn\!\left(\prod_{{\alpha}<\kappa}X_{\alpha}\right)\le \kappa$.}

Concerning $\ps$ we can obtain a slightly stronger result.

\begin{theorem}\label{pro:ps-prod2}
 If $\<X_{\alpha}:{\alpha}<{\kappa}\>$ are $T_0$-spaces, 
$$\ps\!\Big(\prod_{{\alpha}<\kappa}X_{\alpha}\Big)=
\sup\{\ps(X_{\alpha}) : {\alpha}<\kappa\}$$
provided $\sup\{\ps(X_{\alpha}) : {\alpha}<\kappa\}$ 
is  ${\kappa}$-decomposable.
    \end{theorem}

\begin{proof} Write $X=\prod_{{\alpha}<\kappa}X_{\alpha}$ and  ${\delta}=\sup\{\ps(X_{\alpha}) : {\alpha}<\kappa\}$.
Let $\{I_{\alpha}:{\alpha}<\kappa\}$ be a partition of $\delta$ such that 
$\operatorname{tp}(I_{\alpha})=\delta$ for each $\alpha$.  
For each $\zeta\in\delta$ let $\alpha_{\zeta}<\kappa$ be such that 
$\zeta\in I_{\alpha_{\zeta}}$.

Informally, the Seeker’s strategy in the game $\gamexb{X}{{\delta}}$ is to play $\kappa$ parallel games, using the indices in $I_{\alpha}$
to determine the $\alpha$th coordinate of the  point selected by the Hider.

More formally, for each $\alpha<\kappa$ let $\sigma_{\alpha}$ be a winning strategy of 
the Seeker in the game $\gamexb{X_{\alpha}}{\delta}$, and
define the strategy $\sigma$ as follows.
For ${\zeta}\in {\delta}$ let 
\begin{displaymath}
  \mb H({\zeta})={\sigma}_{{\alpha}_{\zeta}}(\mb i\restriction (I_{{\alpha}_{\zeta}}\cap {\zeta})),
    \end{displaymath}  
and     
\begin{displaymath}
    \mb U({\zeta})=  {\sigma}(\mb i\restriction {\zeta})=\{x\in X: x({\alpha}_{\zeta})\in \mb H({\zeta})\}.
\end{displaymath}

Then for each ${\alpha}<{\kappa}$ and ${\zeta}\in I_{\alpha}$, we have $\mb H({\zeta})\in {\tau}^+_{X_{\alpha}}$.
Moreover, the pair $\<\mb H\restriction I_{\alpha},\mb i\restriction I_{\alpha}\>$ codes a play in 
$\gamexb{X_{\alpha}}{{\delta}}$ in which the Seeker followed the strategy ${\sigma}_{\alpha}$.

Hence, $\mb W^{\alpha}=\bocox{\mb H\restriction I_{\alpha}}{\mb i\restriction I_{\alpha}}{{\delta}}$
has at most one element. 

Hence, $|\{x({\alpha}):x\in \mb W\}|\le 1$ for each ${\alpha}<{\kappa}$,
where $\mb W=\bocox{\mb U}{\mb i}{{\delta}}$ is the outcome of the game 
$\gamexb{X}{{\delta}}$.
Hence, $|\mb W|\le 1$, so the Seeker wins which completes the proof of the theorem.
\end{proof}

We finish this section  with an easy observation.

\begin{observation}\label{obs:easy}
$\ps(\prod_{{\alpha}<{\kappa}}(X_{\alpha}))\le  \sum_{{\alpha}<{\kappa}}\ps(X_{\alpha}).$ 
\end{observation}

Indeed, we can imitate the proof of the previous theorem in a simpler setup. 
Namely, partition ${\delta}=\sum_{{\alpha}<{\kappa}}\ps(X_{\alpha})$
into pieces $\{I_{\alpha}:{\alpha}<{\kappa}\}$ such that $\operatorname{tp}(I_{\alpha})=\ps(X_{\alpha})$. Now, 
 the Seeker can use  the steps whose indices in $I_{\alpha}$ to determine the ${\alpha}$th coordinate of the point selected by the Hider  (see the proof of Theorem \ref{pro:ps-prod2}).

\section{$\ps$ and ${\sm}$ for  familiar spaces}\label{sc:familiar}

\begin{theorem}\label{tm:ps=kappa}
    If ${\kappa}$ is a regular cardinal, then 
    \begin{displaymath}
    {\sm}(D(2)^{\kappa})=\ps(D(2)^{\kappa})=\ps(\<D(2)^{\kappa}, {\tau}_{<{\kappa}}\>)=
    {\sm}(\<D(2)^{\kappa}, {\tau}_{<{\kappa}}\>)={\kappa},
    \end{displaymath}
    where ${\tau}_{<{\kappa}}$ denotes the topology generated by the 
    $G_{<{\kappa}}$-sets.  
    \end{theorem}

\begin{proof}
    Since the ${\tau}_{<{\kappa}}$-topology is finer than the product topology,   
    $${\kappa}=\we(D(2)^{\kappa})\ge \ps(D(2)^{\kappa})
    \ge {\sm}(D(2)^{\kappa})\ge {\sm}(\<D(2)^{\kappa}, {\tau}_{<{\kappa}}\>).$$

    To show ${\kappa}\le {\sm}(\<D(2)^{\kappa}, G_{<{\kappa}}\>)$, 
    write $X=\<D(2)^{\kappa}, {\tau}_{<{\kappa}}\>$, and 
    we apply the following argument.

Since $\Delta(X)\ge \we(X)$, the space $X$ is resolvable, that is, there is a dense subset 
$Y\subs X$ such that $X\setm Y$ is also dense. 
        We will show that $\wingamexyb{H}XY{{\beta}}$  for each ${\beta}<{\kappa}$. 

Construct  the winning strategy of the Hider in the game  $\gamexyb{X}{Y}{{\beta}}$
as follows.  
Let $\mc B$ be the family of elementary open subsets of $X$.
Along the game the Hider will choose 
     a decreasing sequence $\<B_{\zeta}: {\zeta}\le{\beta}\>\subs \mc B$ 
      such that 
     \begin{equation*}
     B_{\zeta}\subs 
     \boco{\mathbf U\restriction {\zeta}}{\mathbf i\restriction {\zeta}}.
     \end{equation*}
     First, $B_{0}=X$ works. 
     Assume that we have    $B_{\xi}$ for  ${\xi}<{\zeta}$.

     If ${\zeta}$ is a limit ordinal,     
     we can take  
          $B_{\zeta}= \bigcap\{B_{\xi}:{\xi}<{\zeta}\}\in \mc B$.   

     Assume that ${\zeta}={\xi}\oplus 1$.
     Then the  Seeker chooses 
     $\mathbf U({\xi})\in {\tau}_Z$.

     If $\mathbf U({\xi})\cap B_{\xi}=\empt $, then let 
     $\mathbf i({\xi})=0$ and then 
     $B_{{\zeta}}=B_{\xi}$ works because 
     $$\boco{\mathbf U\restriction {\zeta}}{\mathbf i\restriction {\zeta}}=
     \boco{\mathbf U\restriction {\xi}}{\mathbf i\restriction {\xi}}\setm \mathbf U({\xi})\supset B_{\xi}.$$
     
     If $\mathbf U({\xi})\cap B_{{\xi}}\ne \empt $, then 
      we can choose  $B_{{\zeta}}\in \mc B$ with $B_{{\zeta}}\subs 
      \mathbf U({\xi})\cap B_{{\xi}}$. 

After ${\beta}$ steps, we have  
     \begin{equation*}%\tag{$K_n$}
     B_{{\beta}}\subs  
     \bocox{\mathbf U}{\mathbf i}X.
     \end{equation*}
Since both $Y$ and $X\setm Y$ are dense, the set $B_{\beta}$ intersects both, 
so the Hider wins the game. 

This completes the proof of the Theorem. 
\end{proof}

    \begin{remark} $\ps (\mathbb S^n)={\omega}$ for each $n\ge 1$,  where 
        $\mathbb S$ denotes the Sorgenfrei line, because  the Sorgenfrei topology is finer 
        than the Euclidean topology.
    \end{remark}

    \begin{theorem}\label{tm:double-arrow-power}
    For any cardinal  ${\mu}>0$,
    \begin{displaymath}
    {\ps(\mathbb A^{\mu})}=\left\{\begin{array}{ll}
    {{\omega}\oplus {\mu}}&\text{if ${\mu}\le {\omega}$},\\
    {{\mu}}&\text{{if ${\mu}>{\omega}$}},
    \end{array}\right.
    \end{displaymath}
    and 
    \begin{displaymath}
        {{\sm}(\mathbb A^{\mu})}=\left\{\begin{array}{ll}
        {{\omega}\oplus {1}}&\text{if ${\mu}< {\omega}$},\\
        {{\omega}\oplus {\omega}}&\text{if ${\mu}= {\omega}$},\\
        {{\mu}}&\text{{if ${\mu}>{\omega}$}}.
        \end{array}\right.
        \end{displaymath}
        \end{theorem}

    \begin{proof}
We will identify the elements of $\mbb A^{\mu}$ with the elements of 
$\mbbI^{\mu}\times 2^{\mu}$.
        We start with the following straightforward observation.
    
    \begin{obs}\label{obs:1}
       For each  $\empt\ne U\in{\tau}_{\mbb A^{\mu}}$, 
       there is non-empty $V_U\in {\tau}_{\mbbI^{\mu}}$ such that 
       $V_U\times 2^{\mu}\subs U$.
        \end{obs}

    \begin{lemma}\label{lm:s-n}
    For each cardinal  $1\le {\mu}\le {\omega}$, 
    we have $\wingamexb{S}{\mathbb A^{\mu}}{{\omega}\oplus {\mu}}$.
    \end{lemma}
    
    \begin{proof}
        Let $\mc B=\{B_k:k<{\omega}\}$ be a base of $\mbbI^{\mu}$.
        In the $k$th turn the Seeker asks $\mb U(k)=B_k\times {}^{\mu}2\in {\tau}_{\mathbb A^n}$.
        
        Then $Y=\bocox{\mathbf U\restriction {\omega}}{\mathbf i\restriction {\omega}}{\mathbb A^{\mu}}$ is either empty (and so the Seeker wins), or 
        $Y=\{x\}\times {}^{\mu}2$ for some $x\in \mbbI^{\mu}$, because $\mc B$ is a point separating family in $\mbbI^{\mu}$.
        
        If ${\mu}$ is finite, then 
        $|Y|=2^{\mu}$, and so ${\mu}$ additional steps is enough  for the Seeker to win the game.
        
        If ${\mu}={\omega}$, then 
        $Y$ is homeomorphic to $2^{\omega}$, and so ${\omega}$ additional steps is enough  for the Seeker to win the game.
    \end{proof}

    \begin{lemma}\label{lm:h-n}
        For each natural number $n\ge 1$, 
        we have $\wingamexb{H}{\mathbb A^n}{{\omega}\oplus (n-1)}$.
        \end{lemma}
    
      \begin{proof}
    Using the Observation, the Hider can pick  $i({\alpha})<2$ and non-empty $V_{\alpha}\in {\tau}_{\mbbI^n}$
    for ${\alpha}<   {\omega}$
    such that  $V_{\alpha}\times 2^n\subs \bocox{\mathbf U\restriction {\alpha}+1}{\mathbf i\restriction {\alpha}+1}{\mathbb A^n}$.
    We can also assume that $\overline {V_{{\alpha}+1}}\subs V_{\alpha}$ and 
    diameter of $V_{\alpha}$ is at most $2^{-{\alpha}}$. So 
    there is $x\in \bigcap_{{\alpha}\in {\omega}} V_{\alpha}\subs \mbbI^n$.
    Thus, $\bocox{\mathbf U\restriction {\omega}}{\mathbf i\restriction{\omega}}{\mathbb A^n}$ has 
    at least $2^n$ elements. 
    Thus, the Hider can win in $n-1$ steps: in the ${\omega}\oplus i$th step  the Hider can guarantee that 
    $\bocox{\mb U \restriction {\omega}\oplus i}{\mb I\restriction {\omega}\oplus i}{\mb A^n}$
contains at least $2^{n-i}$ elements.
      \end{proof}

    Lemmas \ref{lm:s-n} and \ref{lm:h-n} prove that 
    $\ps (\mbb A^{\mu})={\omega}\oplus{\mu}$ for $1\le{\mu} \le {\omega}$.
    
    \begin{lemma}\label{lm:s-n-sm}
        For natural number  $n\ge 1$ and $Y\subs \mbb A^n$, 
        we have $\wingamexyb{S}{\mathbb A^n}{Y}{{\omega}\oplus 1}$.
        \end{lemma}
        
    \begin{proof}
        let $\mc B=\{B_k:k<{\omega}\}$ be a base of $\mathbb R^n$.
        In the $k$th turn the Seeker asks $U_k=V_k\times {}^n2\in {\tau}_{\mathbb A^n}$.
        Then $Z=\bocox{\mathbf U\restriction {\omega}}{\mathbf i\restriction {\omega}}{\mathbb A^n}$ is either empty (and so the Seeker wins), or 
        $Z=\{x\}\times {}^n2$, because $\mc B$ is a point separating family in $\mathbb R^n$.
        Since $|Z|=2^n$, one  additional steps is enough  for the Seeker to win the game: choose an open set $\mb U({\omega})$ such that 
        $\mb U({\omega})\cap Z=Y\cap Z$.

    \end{proof}

    \begin{lemma}\label{lm:bern}
        If $Y=\mbbI\times \{0\}$, then         $\wingamexyb{H}{\mathbb A}{Y}{{\omega}}$.
        \end{lemma}

\begin{proof}
    Using the Observation, the Hider can pick  $i({\alpha})<2$ and non-empty $V_{\alpha}\in {\tau}_{\mbbI^n}$
    for ${\alpha}<   {\omega}$
    such that  $V_{\alpha}\times 2^n\subs \bocox{\mathbf U\restriction {\alpha}+1}{\mathbf i\restriction {\alpha}+1}{\mathbb A^n}$.
    We can also assume that $\overline {V_{{\alpha}+1}}\subs V_{\alpha}$ and 
    diameter of $V_{\alpha}$ is at most $2^{-{\alpha}}$. So 
    there is exactly one $x\in \bigcap_{{\alpha}\in {\omega}} V_{\alpha}\in \mbbI$.
    Thus, $\bocox{\mathbf U\restriction {\omega}}{\mathbf i\restriction{\omega}}{\mathbb A}$ has 
    exactly two elements, $\<x,0\>$ and $\<x,1\>$. 
    Thus, the Hider wins.
\end{proof}

Lemmas \ref{lm:s-n-sm} and \ref{lm:bern} imply that
$\sm(\mbb A^n)={\omega}\oplus 1$ for $1\le n<{\omega}$.

\begin{lemma}\label{lm:aomegasm}
    Fix $Y\subs \mbb A^{\omega}$ such that 
for each $x\in [0,1]^{\omega}$ 
and for each $\varepsilon\in \operatorname{Fn}({\omega},2)$
there are $a,b\in 2^{\omega}$ such that 
$\<x,a\>\in Y\cap [\varepsilon]$ and 
$\<x,b\>\in [\varepsilon]\setm Y$.  
Then   $\wingamexyb{H}{\mathbb A^{\omega}}{Y}{{\omega\oplus n}}$
for each $n<{\omega}$.
    \end{lemma}

\begin{proof}
    Using the Observation, the Hider can pick  $i({\alpha})<2$ and non-empty $V_{\alpha}\in {\tau}_{\mbbI^n}$
    for ${\alpha}<   {\omega}$
    such that  $V_{\alpha}\times 2^{\omega}\subs \bocox{\mathbf U\restriction {\alpha}+1}{\mathbf i\restriction {\alpha}+1}{\mathbb A^n}$.
    We can also assume that $\overline {V_{{\alpha}+1}}\subs V_{\alpha}$ and 
    diameter of $V_{\alpha}$ is at most $2^{-{\alpha}}$. So 
    there is exactly one $x\in \bigcap_{{\alpha}\in {\omega}} V_{\alpha}\in \mbbI$.

    Thus, $\mathbb W_{\omega}=\bocox{\mathbf U\restriction {\omega}}{\mathbf i\restriction{\omega}}{\mathbb A}$ is homeomorphic to $2^{\omega}$. 
Moreover, both $Y$ and $\mathbb W_{\omega}\setm Y$ are dense 
in $\mathbb W_{\omega}$.

    The Hider can play $n$ more steps in such a way that 
    the interior of the outcome  $\mathbb W$ of the game is nonempty in $\mathbb W_{\omega}$.
    Since both $Y$ and $\mathbb W_{\omega}\setm Y$ are dense
    in $\mathbb W_{\omega}$, neither $\mathbb W\subseteq Y$ 
    nor $\mathbb W\cap Y=\varnothing$. 
    Hence, the Hider wins the game.
\end{proof}

By Lemma \ref{lm:aomegasm}, $\sm(\mbb A^{\omega})\ge {\omega}\oplus {\omega}$.
Since $\sm(A^{\omega})\le \ps (\mbb A^{\omega})={\omega}\oplus {\omega}$.
So we proved $\sm(\mbb A^{\omega})={\omega}\oplus{\omega}$.

\begin{lemma}\label{lm:kalaom}
    If ${\kappa}>{\omega}$ is a cardinal, then 
    $\ps(\mathbb A^{\kappa})={\sm}(\mathbb A^{\kappa})={\kappa}$.
\end{lemma}

\begin{proof}  
    The Hider  
    can play such a way that for each ${\alpha}<{\kappa}$ there is a 
    function $f_{\alpha}$ with $\dom(f_{\alpha})\in {[{\kappa}]}^{|{\alpha}|+{\omega}}$
    and $\ran(f_{\alpha})\subs {\tau}_{\mbb A}$
    such that 
    \begin{displaymath}
        [f_{\alpha}]\subs 
    \bocox{\mathbf U\restriction     {\alpha}}{\mathbf i\restriction{\alpha}}{X}.
    \end{displaymath}
    
    Hence,  if $\{Y,Z\}$ is a partition of $\mathbb A^{\kappa}$ into  
    $G_{<\kappa}$-dense sets,
    then ${\sm}(Y,\mathbb A^{\kappa})\ge {\kappa}$.
    
    On the other hand,  ${\kappa}$ has a partition into ${\kappa}$ many sets all of them have order type 
    ${\omega}\oplus1$, so $\ps(\mathbb A^{\kappa})\le {\kappa}$. 
    \end{proof}
  This completes the proof of the Theorem \ref{tm:double-arrow-power}. 
    \end{proof}

    In \cite{Scott86}*{Proposition 2.5} Scott proved the following:
    {\em If ${\kappa}$ is a regular cardinal, and $X\subs {\kappa}$ is stationary, then 
    $\psn(X)={\kappa}$.  }
    We prove an analogue result for  $\ps$ and  ${\sm}$:
    \begin{theorem}\label{pr:stat}
        If ${\kappa}$ is a regular cardinal and  $X\subs {\kappa}$ is a stationary set,  then 
        $\ps(X)={\sm}(X)={\kappa}$. 
    \end{theorem}
    
    \begin{proof}
    Pick  a stationary $Y\subs X$ such that $X\setm Y$ is also stationary. 
    We will show that $\wingamexyb{H}{X}{Y}{{\beta}}$ for ${\beta}<{\kappa}$.
    
    Assume that in the ${\alpha}$th step the Seeker asks $U_{\alpha}$.
    Let the Hider respond with  ``1''  iff   $U_{\alpha}$  is co-bounded, i.e. it 
    contains  $X\setm {\mu}$ for some ${\mu}<{\kappa}$
    
    Since an open subset of a stationary set is either non-stationary or co-bounded, 
    $\boou{\mb U({\alpha})}{\mb i({\alpha})}{X}$ contains a subset that is relatively closed  and unbounded in $X$
    for each ${\alpha}<{\beta}$.
    Therefore, $B=\bocox{\mathbf U}{\mathbf i}{X}$ also contains a subset that is relatively closed and unbounded in $X$.
    Thus,  both $Y$ and $X\setm Y$ intersect $B$, i.e., the Hider wins.
    \end{proof}

        \section{The range of $\ps$ and ${\sm}$}\label{sc:range}

\subsection*{Stepping up in $\ps$.}

By Theorem \ref{tm:double-arrow-power} we know that 
${\omega}\oplus {\omega}$ lies in the range of $\ps$. To obtain larger 
countable ordinals in that range, we need the following stepping-up theorem.

\begin{theorem}\label{tm:steppingup}
Let $\<X,{\tau}\>$ be a $T_2$ topological space with 
${\pi}(X)={\omega}$. Then there exists a $T_2$ topological space 
$Z$ such that  $\ps(Z)={\omega}\oplus \ps(X)$ and ${\pi}(Z)={\omega}$.
Moreover, $Z$ can be chosen regular (respectively, $0$-dimensional) whenever $X$ is.
\end{theorem}

\begin{proof}
Let $C$ be the Cantor set. For each $c\in C$ let $\{I^c_n:n<{\omega}\}$ 
be a sequence of pairwise disjoint clopen intervals  in $C$ converging to 
$c$. Let ${\varepsilon}$ be the usual Euclidean topology on $C$.
Fix a ${\pi}$-base $\{P_n:n<{\omega}\}$ of $X$.

The underlying set of $Z$ will be $C\times X$.
For $c\in C$, $W\in {\tau}_X$,  and   $c\in E\in {\varepsilon}$   define 
\begin{equation}\notag%\label{eq:Z}
V(c,W,E)=\Big((\{c\}\times W)\cup\bigcup \{I^c_n\times X: P_n\subs W\}\Big)\cap (E\times X).
\end{equation} 
For $z=\<c,w\>\in C\times Z$  define 
\begin{displaymath}
\mc V_z=\{V(c,W,E): w\in W\in  {\tau}_X, c\in E\in \varepsilon, E\text{ clopen}\}.
\end{displaymath}

It is easy to check that 
the family 
\begin{equation}\notag
\mc V=\{\mc V_z:z\in Z\}
\end{equation}
is a neighborhood system of a topology  ${\tau}_Z$  on $Z$ (see \cite{En89}*{Proposition 1.2.3}).
Moreover, observe that 
\begin{enumerate}[(1)]
\item For each $E\in {\varepsilon}$, we have $(E\times X)\in {\tau}_Z$.  
\item If $\mc E$ is a ${\pi}$-base in $C$, then $\{E\times X: E\in \mc E\}$ is a ${\pi}$-base in $Z$.
Hence, $\pi(Z)={\omega}$. 
\end{enumerate}

Since \begin{displaymath}
\overline{V(c,W,E)}=V(c,W,E)\cup (\{c\}\times \overline{W}),
\end{displaymath}
the space $Z$ is regular  or 0-dimensional whenever
$X$ is  regular  or 0-dimensional.

Next we show $\ps(Z)\ge {\omega}\oplus \ps(X)$.

Fix ${\alpha}<\ps(X)$, and fix a strategy ${\sigma}$ for Seeker 
in the game $\gamexb{Z}{{\omega}+{\alpha}}$.
To show that ${\sigma}$ is not a winning strategy, 
let Hider play the following strategy in the first ${\omega}$ move of that game
against ${\sigma}$. 

Along the game the Hider will choose non-empty clopen sets $K_n\in {\varepsilon}$ 
for $n=0,1,2,\dots$ such that 
\begin{equation}\notag%\label{eq:o+a}%\tag{$K_n$}
(K_n\times X)\in \boco{\mathbf U\restriction n}{\mathbf i\restriction n}.
\end{equation}
First, $K_{0}=C$ works. 

Assume that we have $K_n$ and the Seeker chooses 
$\mathbf U(n)={\sigma}(\mathbf i\restriction n)\in {\tau}_Z$ according the  strategy ${\sigma}$. 

If $\mathbf U(n)\cap (K_n\times X)=\empt $, then let $\mathbf i(n)=0$ and then 
$K_{n+1}=K_n$ works because 
$$\boco{\mathbf U\restriction n+1}{\mathbf i\restriction n+1}=\boco{\mathbf U\restriction n}{\mathbf i\restriction n}\setm \mathbf U(n)\supset K_n\times X.$$

If $\mathbf U(n)\cap (K_n\times X)\ne \empt $, then let $\mb i(n)=1$ and choose a non-empty clopen $K_{n+1}\in {\varepsilon}$ with $K_{n+1}\subs K_n$
such that $K_{n+1}\times X\subs \mathbf U(n)$. 

After ${\omega}$ steps, $\bigcap_{n\in {\omega}} K_n\ne \empt$, so we can pick $c\in \bigcap_{n\in {\omega}} K_n.$

Then 
\begin{displaymath}
\{c\}\times X\subs \boco{\mathbf U\restriction {\omega}}{\mathbf i\restriction {\omega}}.
\end{displaymath}

We will play  an additional game  coded by $\<\mathbf S,\mathbf j\>$ in     
$\gamexb{X}{{\alpha}}$ played by $X$-Seeker and $X$-Hider as follows: 

For ${\xi}<{\alpha}$, if  $\mathbf U({\omega}\oplus{\xi})={\sigma}(\mathbf i\restriction 
{\omega}\oplus{\xi})\in {\tau}_Z$, then in the ${\xi}^{th}$ turn of the 
game $\gamexb{X}{{\alpha}}$
the  X-Seeker plays $\mathbf S({\xi})\in {\tau}_X$ such that 
$$\{c\}\times \mathbf S({\xi})=(\{c\}\times X)\cap \mathbf U({\omega}\oplus{\xi}).$$
In this way we described a strategy for $X$-seeker in ${\alpha}$ steps.
Since the $X$-Seeker can not win in ${\alpha}$ steps using any strategy,
there is $\mathbf j$ such that $|\bocox{\mathbf S}{\mathbf j}{X}|\ge2$. 
Let the Hider play $\mb i({\omega}\oplus {\xi})=\mb j({\xi})$ in the original game after the first ${\omega}$ steps.

However, $$\bocox{\mathbf U}{\mathbf  i}{X}\cap (\{c\}\times Y)=\{c\}\times \bocox{\mathbf S}{\mathbf j}{X}.$$

Thus, Seeker lost the game  $\gamexb{Z}{{\omega}+{\alpha}}$. 
Hence,  the strategy ${\sigma}$  was not a winning strategy, 
and we conclude  $\ps(Z)\ge{\omega}\oplus \ps(X)$.

The reverse inequality
$\ps(Z)\le{\omega}\oplus \ps(X)$ is straightforward. 
In the first ${\omega}$ moves of the  game  the Seeker can ensure that 
\begin{displaymath}
    \bocox{\mathbf U\restriction {\omega}}{\mathbf i\restriction {\omega}}{Z}
    \subs \{c\}\times X
\end{displaymath}
for some $c\in C$ simply  by playing  $B_n\times X$
on the $n$th turn, where $\{B_n:n<{\omega}\}$ is  a base of $\mc C$. 
Then, in the subsequent  $\ps(X)$-many turn, the Seeker 
can treat the game as if it occurs inside e $\{c\}\times X$, 
and hence  the Seeker can win.  
\end{proof}
\begin{corollary}\label{cor:addition}
For each $n,m<{\omega}$ there is a topological space $X_{n,m}$ such that 
$\ps(X_{n,m})=({\omega}\otimes n)\oplus m$, where 
    $\otimes$ denotes  ordinal multiplication. 

    Hence, $\ps(\sum_{n\in {\omega}}X_{n,0})={\omega}\otimes {\omega}$. 
\end{corollary}

\begin{proof}
By induction on $n$.

If $n=0$, then $X_{0,m}=D(2^m)$ satisfies the requirements. 

If we have $X_{n,m}$, then apply Theorem \ref{tm:steppingup}
for $X=X_{n,m}$ to obtain $X_{n+1,m}$ as $Z$.
\end{proof}

\subsection*{Resolvability and ${\sm}$.}    
\begin{theorem}\label{tm:resolvable}
    If $X$ is a resolvable space, then ${\sm}(X)\ge \omega$.
    \end{theorem}
    
    \begin{proof}
    Let $\langle Y,Z\rangle$ be a partition of $X$ into two dense subsets. 
    Fix $n<\omega$ and consider the game $\gamexyb{X}{Y}{n}$. 
    The Hider can play in such a way that 
    the interior of the outcome of the game is nonempty, that is, 
    \begin{displaymath}
        V=\operatorname{int}(\bocox{\mb U}{\mb i}{X})\ne\varnothing.
\end{displaymath}
    Since both $Y$ and $Z$ are dense, we have $V\cap Y\ne\varnothing$ 
    and $V\cap Z\ne\varnothing$. 
    Therefore, neither $\bocox{\mb U}{\mb i}{X}\subseteq Y$ 
    nor $\bocox{\mb U}{\mb i}{X}\cap Y=\varnothing$. 
    Hence, the Hider wins the game.
    \end{proof}
    
    \begin{corollary}If   ${\sm}(X)< {\omega}$, then $X$ is OHI, i.e., $X$ has 
        no resolvable open subspace.    \end{corollary}

The following notion was introduced in \cite{He43} under the name MI or MI-space.
\begin{definition}\label{df:submaximal}
A crowded topological space is called {\em submaximal} if every dense subset is open. 
\end{definition}

We will use the following  statement:
\begin{prop}[\cite{JuSoSz06}*{Theorem 4.1}]\label{pr:subm}
For each infinite ${\kappa}$, there is a 0-dimensional submaximal
space with $|X|=\Delta(X)\ge {\kappa}$. 
\end{prop}

\begin{theorem}\label{tm:submaximal=sm}
If $X$ is submaximal, then ${\sm}(X)= 2$ and $\ps(X)=\log(|X|)$. 
\end{theorem}

\begin{proof}
 Let $Y\subs X$.   
 By \cite{ArCo95}*{Theorem 1.2(c)}, the set 
$F=Y\setm \operatorname{int}(Y)$ is closed in $X$.
Hence, $Y$ is the disjoint union of the open set $U=\operatorname{int}(Y)$ and the closed set 
$F=Y\setm U$. Let the Seeker  play $\mb U(0)=U$ and $\mb U(1)=X\setm F$.    
Then the possible outcomes of the game are $U,F,X\setm (U\cup F)$ and $ \empt$.
In each case, the Seeker wins. Therefore, $\sm(Y,X)\le 2$.

Since $X$ has a subset $Z$ that is neither open nor closed, we have $\sm(Z,X)>1$; hence  $\sm(X)=2$.

Finally, observe that $\ps(X)=\min({\sm}(X),\log(|X|))=\log(|X|)$ by Theorem \ref{tm:pssmlog}.
\end{proof}

By combining Theorems~\ref{pr:subm} and \ref{tm:submaximal=sm}, we obtain the following.

\begin{corollary}\label{cr:subm}
    For each infinite cardinal $\kappa$, there exists a $0$-dimensional $T_2$ space $X$ such that  
    $X$ is a dense subset of the space $D(2)^{2^{\kappa}}$, 
    \begin{displaymath}
    |X| = \Delta(X) = \kappa, \quad\ps(X)=\log({\kappa}) \quad \text{and} \quad {\sm}(X) = 2.
    \end{displaymath}   
\end{corollary}

\begin{observation}\label{obs:subproduct}

If  $X$ is a submaximal space with $|X|=\Delta(X)$, then ${\sm}(X)=2$,
but $X\times X$ is resolvable, so ${\sm}(X^2)\ge {\omega}$.

\end{observation}

\section{Developable spaces}\label{sc:developable}

Recall that a countable collection $\mbb U$ of families of open sets in a space $X$ is called 
{\em quasi-development} if and only if $\{\operatorname{St}(x, \mc U) : \mc U \in \mbb U\}$ 
forms a local base at $x$ for every $x \in X$, where 
\begin{displaymath}
\operatorname{St}(x, \mc U) = \bigcup \{U \in \mc U : x \in U\}.
\end{displaymath}
A quasi-development $\mbb U$ is called a {\em development} if and only if every $\mc U \in \mbb U$ is a cover of $X$.  
A space is said to be {\em quasi-developable} (respectively, {\em developable}) if it admits a quasi-development (respectively, a development).

In 1986 Scott obtained the following result:

\begin{theorem}[{\cite{Scott86}*{Theorem 1.7}}]\label{tm:scott} If a space $X$ is developable, 
    then $\psn( X) = \pso(X ) =
\log \we ( X ) = \log (|X |)$.
\end{theorem}

\begin{corollary}\label{cr:metric}
If $X$ is a metric space, then $\ps(X)= \log (|X |)$.
\end{corollary}

\begin{proof}
First, $\ps(X)\le \psn( X)=\log(|X|)$ by Theorem \ref{tm:scott}.
On the other hand, $\log(|X|)\le \ps(X)$ by Theorem \ref{tm:ps2psn}(a).  
\end{proof}

Developable spaces are clearly quasi-developable as well, but the converse does not hold: 
in \cite{Be71}*{Example 2.2}, Bennett showed that the Michael line  (see \cite{Mi63}) separates these two classes of spaces.
Hence, the following result -- although  inspired by Theorem \ref{tm:scott} -- does not follow from it. 

\begin{theorem}\label{tm:quasi}If $X$ is a $T_0$ quasi-developable space, then $\ps(X ) = \log | X |$.
\end{theorem}

Before proving this result, we recall one more definition. 
An open cover $\mc U$ of a space $X$ is called a 
{\em $\theta$-base} of $X$ if there exist families $\mc U_n \subset \mc U$ for each $n < \omega$ such that 
\begin{displaymath}
\mc U = \bigcup_{n < \omega} \mc U_n,
\end{displaymath}
and for every $x \in X$ and every open neighborhood $V$ of $x$, there is some $n < \omega$ such that 
the order of $\mc U_n$ at $x$, defined by
\begin{displaymath}
\operatorname{ord}(x, \mc U_n) = |\{U \in \mc U_n : x \in U\}|,
\end{displaymath}
is finite, and moreover, there exists $U \in \mc U_n$ with $U \subset V$.  

Bennett and Lutzer \cite{BeLu72}*{Theorem 8} proved that a space is quasi-developable if and only if it has a $\theta$-base.

\begin{proof}
 Let $\mc B$ be a base for $X$ such that $|\mc B| = \we( X )$. Now let 
$\mc U =
 \bigcup
 \{\mc U_n : n < {\omega} \}$ be
a ${\theta}$-base for $X$ such that:
\begin{enumerate}[(1)]
\item 
 $\mc U {\subset} \mc B$, hence for each $n < {\omega}$, $|\mc U_n | \le \we(X )$;
\item  $\{\mc U_n : n < {\omega}\}$ is a partition of $U$ that witnesses that $U$ is a ${\theta}$-base, i.e., for each
$x {\in} X$ and for each $V\in {\tau}(x,X)$, there is some $n < {\omega}$ such that $\operatorname{ord} ( x, {\mc U_n} ) < {\omega}$ and
there is some $U {\in} U_n$ such that $x {\in} U {\subset} V.$
\end{enumerate}

For every $m, n < {\omega}$, let $$\mc V_n^m := \Big\{
 \bigcap\{U {\in} {\mc U_n} : x {\in} U \} : x {\in} X \text{ and } \operatorname{ord} (x, {\mc U_n} ) = m\Big\}.$$
Notice that $|\mc V_n^m | \le  |[{\mc U_n} ]^m | \le \we( X)$. Thus, for all $x {\in} X$ and $V {\in} \tau(x, X)$, there are
$m, n < {\omega}$ such that $\operatorname{ord} (x, {\mc U_n} ) = m$ and there is some $U {\in} {\mc U_n}$ such that $x {\in} U {\subset} V$.
Hence, there is $W = \bigcap
\{U {\in} \mc U_n : x {\in} U \} {\in}\mc V_n^m$ such that $x {\in} W {\subset} U {\subset} V$. 
Observe that  $\operatorname{ord}( x, \mc V^n_m ) = 1$.

For each $m, n < \omega$, let $S_n^m\subs \mc P(\mc V_n^m)$ be a separating family on $\mc V_n^m$ such that 
\begin{displaymath}
|S_n^m| \leq \log \we(X).
\end{displaymath}
Define 
\begin{displaymath}
\mc W_n^m = \bigcup \{S : S \in S_n^m\},
\end{displaymath}
so that $|\mc W_n^m| \leq \log \we(X)$.  
Let 
\begin{displaymath}
\mc W = \bigcup \{\mc W_n^m : m, n < \omega\},
\end{displaymath}
and fix an enumeration 
\begin{displaymath}
\mc W = \{W_\alpha : \alpha < \log \we(X)\}.
\end{displaymath}

Let $\sigma$ be the strategy in the game $\gamexb{X}{\log(|X|)}$ for the Seeker in which, at inning $\alpha < \log \we(X)$, the Seeker plays $\mb W({\alpha})=W_\alpha$.  
Suppose that $\sigma$ is not winning. Then there exists a function $\mb i : \log \we(X) \to 2$ such that $\bocox{\mb W}{\mb i}{X})$ has at least two elements, say $x$ and $y$.  

Since $X$ is $T_0$, there exists an open set $U$ such that, without loss of generality, $x \in U$ but $y \notin U$.  
Choose $m, n < \omega$ such that 
\begin{displaymath}
\operatorname{ord}(x, \mc W_n^m) = 1
\quad\text{and}\quad 
\exists\, V_x \in \mc V_n^m \;\; (x \in V_x \subset U).
\end{displaymath}

Now we distinguish two cases.

If $y \notin \bigcup \mc V_n^m$, then choose $S \in S_n^m$ with $V_x \in S$, and let 
$\alpha < \log \we(X)$ be such that $W_\alpha = \bigcup S$.  
Otherwise, there exists $V_y \in \mc V_n^m$ such that $y \in V_y$. Since $V_x$ is the only element of 
$\mc V_n^m$ containing $x$, and $y \notin V_x$, it follows that $V_x \neq V_y$. Hence, there exists 
$S \in S_n^m$ such that $V_x \in S$ but $V_y \notin S$, and again we choose 
$\alpha < \log \we(X)$ with $W_\alpha = \bigcup S$.

Thus, in either case, we have $x \in W_\alpha$ but $y \notin W_\alpha$.  
Since 
\begin{displaymath}
x \in \bocox{\mb W}{\mb i}{X} = \bigcap_{\beta < \log \we(X)} 
\boou{\mb W({\beta})}{\mb i({{\beta}})}{X} \subset 
\boou{\mb W({\alpha})}{\mb i({\alpha})}{X},
\end{displaymath} 
it follows that $\boou{\mb W(\alpha)}{\mb i(\alpha)}{X}= W_\alpha$, and therefore $\mb i(\alpha)= 1$.  
On the other hand, $y \in \boou{B_{\alpha}}{i_{\alpha}}{X}$, so 
$\boou{\mb W(\alpha)}{\mb i(\alpha)}{X} = X \setminus W_\alpha$ and hence $\mb i(\alpha)= 0$.  
This contradiction completes the argument.

Hence, the Seeker has a winning strategy in the game $\gamexb{X}{\log(|X|)}$.  Thus, 
$\ps(X)\le \log(|X|)$.  Since $\ps(X)\ge \log(|X|)$ by Theorem \ref{tm:ps2psn}(a),
we finished the proof of the theorem.  
\end{proof}

\begin{theorem}\label{tm:metric-sm}
    For each cardinal ${\kappa}$ there is a metric space 
    $X$ such that ${\sm}(X)= \log(|X|)\ge {\kappa}$.
    \end{theorem}

    \begin{proof}
    Let ${\lambda}>{\kappa}$ be a strong limit cardinal 
    with cofinality ${\omega}$, and consider the Baire space $X=B({\lambda})$.
    Its underlying set is ${\lambda}^{\omega}$,
    so $|X|=2^{\lambda}={\lambda}^{\omega}$. Since ${\lambda}$ is a strong limit cardinal,
    $\log(|X|)={\lambda}$. Moreover, $\we(B({\lambda}))={\lambda}$.   
    
Hence, we can apply Corollary \ref{tm:smlarge},  taking  ${\kappa}={\lambda}$ and  $X=B({\lambda})$
to conclude that $$\sm(B({\lambda}))={\lambda}.$$
    \end{proof}

\section{Problems}

In the Introduction we raised  three problems:
\alphaleomegaone*
{\exactsum*}
{\productsm*}

In \cite{Scott86}*{Corollary 2.4} Scott proved that 
$\psn(S)={\omega}_1$ for the Suslin $S$.
\begin{restatable}{problem}{suslin}\label{cr:susline}
    Determine $\ps(S)$ and ${\sm}(S)$ for the Suslin-line.
\end{restatable}

Observe that
$\ps(S)>{\omega}$   for a Suslin line $S$ by Theorem \ref{tm:ps2psn}. But that theorem does not exclude
that $\ps(S)={\alpha}$ for some ${\omega}<{\alpha}<{\omega}_1$.
\medskip
\begin{restatable}{problem}{ostar}\label{pro:ostar}
    Determine $\ps({\omega}^*)$ and ${\sm}({\omega}^*)$.
    \end{restatable}

\medskip
If $X$ is a door space --- in particular, if $X$ has at most one accumulation point ---
  then ${\sm}(X)=1$. If $X$ is submaximal, then ${\sm}(X)=2$.
\begin{problem}
Is it possible for a crowded space $X$ to satisfy ${\sm}(X)=3$? 
More generally, can ${\sm}(X)$ take any arbitrary natural number, when $X$
is a crowded space?
\end{problem}

\medskip

In Theorem \ref{tm:metric-sm} 
we computed ${\sm}(B({\lambda}))$ for the  space $B({\lambda})$,
but only for certain values of ${\lambda}$.  
This limitation suggests the following natural question:
\begin{problem}
Compute ${\sm}(B({\kappa}))$ for every Baire space $B({\kappa})$.
\end{problem}

\medskip

The next problem appears to be both the most difficult and the most interesting:

\begin{problem}\label{df:determined}
Are the games $\gamexb{X}{{\beta}}$ and $\gamexyb{X}Y{{\beta}}$
determined?
    \end{problem}

 \begin{conj}

    {Let $B\subs \mbbI$ be a Bernstein set, and 
    let $\mbb A$ denote the double arrow space. 
    Define \begin{displaymath}
    Y=
    \{\<x,1\>\in \mbb A: x\in  B\}.
    \end{displaymath}
    Show that $\gamexyb{\mbb A}Y{{{\omega}}}$ is not determined. 
    } 
 \end{conj}


\begin{bibdiv}

\begin{biblist}
      

    \bib{Ar78}{article}{
   author={Arhangel\cprime ski\u{\i}, A. V.},
   title={The structure and classification of topological spaces and
   cardinal invariants},
   language={Russian},
   journal={Uspekhi Mat. Nauk},
   volume={33},
   date={1978},
   number={6(204)},
   pages={29--84, 272},
   issn={0042-1316},
   %review={\MR{526012}},
}


\bib{ArCo95}{article}{
   author={Arhangel\cprime ski\u{\i}, A. V.},
   author={Collins, P. J.},
   title={On submaximal spaces},
   journal={Topology Appl.},
   volume={64},
   date={1995},
   number={3},
   pages={219--241},
   issn={0166-8641},
%   review={\MR{1342519}},
   %doi={10.1016/0166-8641(94)00093-I},
}
		


    \bib{Be71}{article}{
        author={Bennett, Harold R.},
        title={On quasi-developable spaces},
        journal={General Topology and Appl.},
        volume={1},
        date={1971},
        number={3},
        pages={253--262},
        issn={0016-660X},
    %    review={\MR{288725}},
     }
             
     \bib{BeLu72}{article}{
        author={Bennett, H. R.},
        author={Lutzer, D. J.},
        title={A note on weak $\theta $-refinability},
        journal={General Topology and Appl.},
        volume={2},
        date={1972},
        pages={49--54},
        issn={0016-660X},
   %     review={\MR{301697}},
     }

     \bib{En89}{book}{
        author={Engelking, Ryszard},
        title={General topology},
        series={Sigma Series in Pure Mathematics},
        volume={6},
        edition={2},
        note={Translated from the Polish by the author},
        publisher={Heldermann Verlag, Berlin},
        date={1989},
        pages={viii+529},
        isbn={3-88538-006-4},
       % review={\MR{1039321}},
     }

     \bib{He43}{article}{
   author={Hewitt, Edwin},
   title={A problem of set-theoretic topology},
   journal={Duke Math. J.},
   volume={10},
   date={1943},
   pages={309--333},
   issn={0012-7094},
  % review={\MR{8692}},
}

\bib{JuSoSz06}{article}{
   author={Juh\'{a}sz, Istv\'{a}n},
   author={Soukup, Lajos},
   author={Szentmikl\'{o}ssy, Zolt\'{a}n},
   title={$\scr D$-forced spaces: a new approach to resolvability},
   journal={Topology Appl.},
   volume={153},
   date={2006},
   number={11},
   pages={1800--1824},
   issn={0166-8641},
 %  review={\MR{2227028}},
   %doi={10.1016/j.topol.2005.06.007},
}




     \bib{Mi63}{article}{
        author={Michael, E.},
        title={The product of a normal space and a metric space need not be
        normal},
        journal={Bull. Amer. Math. Soc.},
        volume={69},
        date={1963},
        pages={375--376},
        issn={0002-9904},
 %       review={\MR{152985}},
    %    doi={10.1090/S0002-9904-1963-10931-3},
     }





    \bib{Scott86}{article}{
        author={Scott, Brian M.},
        title={The pseudoweights of a space},
        journal={Topology Appl.},
        volume={23},
        date={1986},
        number={1},
        pages={91--101},
        issn={0166-8641},
 %       review={\MR{849097}},
     %   doi={10.1016/0166-8641(86)90020-9},
     }
     
     
     \bib{Ho74}{article}{
        author={Hodel, R. E.},
        title={On the weight of a topological space},
        journal={Proc. Amer. Math. Soc.},
        volume={43},
        date={1974},
        pages={470--474},
        issn={0002-9939},
%        review={\MR{334125}},
%        doi={10.2307/2038916},
     }


\end{biblist}





\end{bibdiv}
\end{document}